\documentclass[11pt, a4paper, oneside]{amsart}
\usepackage{amsmath,amsthm,amsrefs}
\usepackage[T1]{fontenc}
\usepackage[utf8]{inputenc}
\usepackage{graphicx,mathrsfs,amssymb,color,comment}
\usepackage{enumitem}
\usepackage{multicol}

\usepackage{tikz}
\usepackage{xcolor}
\usetikzlibrary{quotes,angles}

\newtheorem{theorem}{\textbf{Theorem}}[section]
\newtheorem{proposition}[theorem]{\textbf{Proposition}}
\newtheorem{lemma}[theorem]{\textbf{Lemma}}
\newtheorem{corollary}[theorem]{\textbf{Corollary}}

\theoremstyle{definition}
\newtheorem{definition}[theorem]{\textbf{Definition}}

\theoremstyle{remark}

\newtheorem{example}[theorem]{Example}

\allowdisplaybreaks
\numberwithin{equation}{section}
\let\inf\undefined\DeclareMathOperator*{\inf}{inf\vphantom{p}}
\let\min\undefined\DeclareMathOperator*{\min}{min\vphantom{p}}
\let\max\undefined\DeclareMathOperator*{\max}{max\vphantom{p}}

\newcommand{\mul}[3]{\mathrm{Mul}_{#3}(#1,#2)} 
\newcommand{\R}{\ensuremath{\mathbb R}} 
\newcommand{\N}{\ensuremath{\mathbb N}} 
\newcommand{\prob}{\ensuremath{\mathscr{P}}} 

\renewcommand{\d}{\ensuremath{\,\mathrm d}} 
\newcommand{\weakto}{\ensuremath{\rightharpoonup}} 

\newcommand{\eeta}{\boldsymbol \eta}
\newcommand{\ev}{\mathsf e} 
\newcommand{\eps}{\varepsilon}

\newcommand{\E}{\ensuremath{\mathscr E}} 
\renewcommand{\L}{\ensuremath{\mathscr L}} 
\newcommand{\G}{\ensuremath{\mathscr G}} 

\title[Optimal control for multiagent systems with aggregation]{Optimal control for multiagent systems with simultaneous aggregation}

\author{Mauro Bonafini}
\address{Mauro Bonafini: University of Verona, Department of Computer Science, Strada Le Grazie 15, I-37134 Verona (Italy)}
\email{mauro.bonafini@univr.it}

\author{Giulia Cavagnari}
\address{Giulia Cavagnari: Politecnico di Milano, Dipartimento di Matematica, Piazza Leonardo Da Vinci 32, I-20133 Milano (Italy)}
\email{giulia.cavagnari@polimi.it}

\author{Antonio Marigonda}
\address{Antonio Marigonda: University of Verona, Department of Computer Science, Strada Le Grazie 15, I-37134 Verona (Italy)}
\email{antonio.marigonda@univr.it}
 \thanks{\emph{Acknowledgments.} This research is supported by INdAM-GNAMPA under the Project 2024 ``Controllo ottimo infinito dimensionale: aspetti teorici ed applicazioni'', CUP E53C23001670001.}		
\begin{document}

\maketitle

\begin{abstract}
In this paper, we introduce an optimal control problem for multi-agent systems with non-local cost which favors simultaneous aggregation of particles. This is done introducing a time-dependent notion of multiplicity whose intrinsic dynamical nature differs from more established geometric-like definitions.

\end{abstract}

\section{Introduction}
The optimal control of multi-agent systems in Wasserstein spaces of probability measures has gained significant interest (cf. e.g. \cite{AJZ,BF,CLOS,JMQ}) due to its applications in swarm robotics, crowd dynamics, and data aggregation. In this work, we study an optimal control problem where the cost functional explicitly favors simultaneous aggregation of particles by introducing a time-dependent notion of multiplicity.
The presence of multiplicity in the cost enforces synchronization in the occupation of spatial points over time, leading to novel aggregation effects and forcing, if the time frame allows, sticky particles evolution as shown in Example \ref{example}. This behavior is particularly desirable e.g. in applications to car sharing or mailing problems.

Unlike traditional approaches where costs are expressed in dependance on solutions of a continuity inclusion determining the admissible trajectories in $\prob(\R^d)$, our formulation considers a probability measure on agents/particles' paths (lifted approach), leveraging the superposition principle from optimal transport theory (cf. \cite{CLOS,CMR}).
We refer to \cite{CMR} for a full-correspondence between these objects, in particular for conditions that ensure the transfer of desired properties to the lifted approach.

We establish, with sufficiently general assumptions, the lower semicontinuity of more general functionals that incorporate non-locality, also in terms of multiplicity of points of $\R^d$, as well as global and instantaneous velocities of individual paths. Up to our knowledge, this result in such a general framework is new in the literature.

Additionally, we prove the existence of minimizers and establish a Dynamic Programming Principle for the value function of our cost with a time-localized Lagrangian, providing a rigorous foundation for optimal control in this setting. The main results are contained in Sections \ref{sec:existence-opt},\ref{sec:DPP}.

\section{Preliminaries and settings}
Let $(X, d_X)$ and $(Y, d_Y)$ be complete separable metric spaces. Following \cite{AGS}, we denote by $\prob(X)$ the space of probability measures on $X$ endowed with the topology of narrow convergence and write $\mu_n \weakto \mu$ if the sequence $\{\mu_n\}_n\subset\prob(X)$ converges to $\mu$ in duality with the set of continuous bounded functions. For future use, we recall that narrow convergence is induced by a distance on $\prob(X)$ (cf. Remark 5.1.1 in \cite{AGS}) and that $\prob(X)$ is separable. 
Given $\mu \in \prob(X)$ and a Borel function $f:X\to Y$, the push-forward of $\mu$ through $f$ is denoted by $f_\sharp \mu\in\prob(Y)$.
The subset of measures in $\prob(X)$ with finite second moment, endowed with the Wasserstein distance $W_2$, is denoted by $\prob_2(X)$.

For $I=[a,b]$, $a \le b$, consider $\Gamma_I:=C^0(I;\R^d)$ endowed with the uniform norm. 
Given $t\in I$, we define the linear and continuous map $\ev_t \colon \Gamma_I \to \R^d, \gamma \mapsto \gamma(t)$ and the map $\dot \ev_t \colon \Gamma_I \to \R^d$ given by
{\small
\[
\dot \ev_t(\gamma):=
\begin{cases}
	\displaystyle\lim_{s\to t}\dfrac{\gamma(s)-\gamma(t)}{s-t}=\lim_{s\to t} \dfrac{\ev_{s}(\gamma)-\ev_{t}(\gamma)}{s-t},
	&\textrm{ if $\gamma$ is differentiable at $t$,}\\
	0,&\textrm{ otherwise}.\end{cases}
\]}
The map $\dot\ev_t(\cdot)$ is Borel measurable and $\dot \ev_t(\gamma)=\dot\gamma(t)$ when $\gamma$ is differentiable at $t$.

\medskip
We now present the setting under consideration. 
\begin{definition}[Multiagent dynamics]\label{def:dynamics}
	Let $I=[a,b]$, $a\le b$. Consider a set-valued map $F:I\times \R^d\times \prob_2(\R^d)\rightrightarrows \R^d$. 
	Given a curve $\mathbf{\theta}=\{\theta_t\}_{t\in I}\subset\prob_2(\R^d)$, let $S^{\mathbf{\theta}}_{I}:\R^d\rightrightarrows \Gamma_I$ and $\boldsymbol{S}_{I}:\prob_2(\R^d) \rightrightarrows \prob(\Gamma_I)$ be the set-valued maps defined by
{\small
	\begin{align*}
		S^{\mathbf{\theta}}_{I}(x):=& \{\gamma\in AC(I;\R^d):\, \dot\gamma(t)\in F(t,\gamma(t), \theta_t)\textrm{ for a.e. $t\in I$},\,\gamma(a)=x\},\\
		\boldsymbol S_{I}(\mu):=& \left\{\eeta\in\prob(\Gamma_{I}):\,\mathrm{supp}\,\big((\ev_a,\mathrm{Id})_\sharp\eeta\big)\subseteq\mathrm{graph}\left(S^{\{{\ev_t}_\sharp\eeta\}_{t\in I}}_{I}\right),\,\ev_{a\sharp}\eeta=\mu\right\}.
	\end{align*}
}
	We assume that $\boldsymbol S_{I}(\cdot)$ has values in $\prob_2(\Gamma_I)$ and its images are $W_2$-compact. These properties are satisfied in particular if $F$ has non-empty compact convex images, it is measurable w.r.t. $t\in I$, Lipschitz continuous w.r.t. $(x,\theta)\in\R^d\times\prob_2(\R^d)$ uniformly w.r.t. $t\in I$ and satisfies an appropriate growth condition. This is discussed e.g. in Section 3 of \cite{CMR} (cf. Lemma 3.5, Corollaries 3.6,3.7).
\end{definition}

\begin{definition}[Multiplicity]\label{def:mult}
	Given $t\in I$, $y\in \R^d$ and $\eeta\in\prob(\Gamma_I)$, we define the \emph{multiplicity} of $\eeta$ at $y$ at time $t$ by
	\[
	\mul{t}{y}{\eeta}:=\eeta\left(\ev_t^{-1}(y)\right)=\eeta\left(\{\gamma\in \Gamma_I : \, y=\gamma(t)\}\right) = ({\ev_t}_\sharp\eeta)\left(\{y\}\right).
	\]
\end{definition}
The multiplicity counts the total mass that passes through the point $y$ at the given time $t$. In more established measure theoretic definitions of multiplicity (cf. e.g. \cite{BCM}), a point is assigned higher multiplicity even if curves visit it at different times. On the other hand, here we stick with a more dynamical perspective and use a localized in time definition, assigning higher multiplicity only for simultaneous occupations of the point $y$ at time $t$.
We now define the cost involved in our optimization problem.
\begin{definition}[Cost functional and variational problem]\label{def:cost}
	Let $I=[a,b]$, $a\le b$. Consider a sequentially l.s.c. Borel function $G \colon \R^d\times\prob(\R^d)\to [0,+\infty]$ and a Borel function $L:I\times \R^d \times \R^d \times \prob(\R^d)\to [0,+\infty]$ such that
	\begin{itemize}
		\item $L(t,\cdot,\cdot,\cdot)$ is sequentially l.s.c. for a.e. $t\in I$,
		\item $L(\cdot,x,z,\mu)$ is Lebesgue measurable for all $(x,z,\mu)\in \R^d\times \R^d \times\prob(\R^d)$,
		\item $L(t,x,\cdot,\mu)$ is convex for all $(t,x,\mu)\in I\times \R^d \times\prob(\R^d)$.
	\end{itemize}
	Given two l.s.c. non-increasing functions $\psi,\phi \colon [0,1]\to [0,+\infty]$ and $\eeta\in\prob(\Gamma_I)$, we define the costs
{\small
	\begin{align*}
		\L_\psi(\eeta) :=&\int_{\Gamma_I} \int_I \psi(\mul{t}{\gamma(t)}{\eeta})\cdot L(t,\gamma(t),\dot\gamma(t),{\ev_t}_\sharp\eeta)\d t \d\eeta(\gamma),\\
		\G_\phi(\eeta) :=&\int_{\Gamma_I}\phi(\mul{b}{\gamma(b)}{\eeta})\cdot G(\gamma(b),\ev_{b\sharp}\eeta)\d\eeta(\gamma), \\
		\E_{\psi,\phi}(\eeta) :=&\L_\psi(\eeta)+\G_\phi(\eeta).
	\end{align*}
}	
	Given $T>0$ and $\mu\in\prob_2(\R^d)$, we are interested in the following problem:
	\[
	\text{minimize the functional $\E_{\psi,\phi}(\eeta)$ over all $\eeta\in \boldsymbol S_{[0,T]}(\mu)$}.
	\tag{P}
	\label{eq:prob}
	\]
\end{definition}

In the next example we show the impact of the time horizon $T$ in the structure of minimizers of \eqref{eq:prob}, due to our choice of multiplicity.
\begin{example}\label{example}
	Let $T>0$, and $I=[0,T]$. In $\mathbb R^2$, consider the points $P_{\pm 1}=(0,\pm 1)$ and $Q_\alpha=(\alpha,0)$ for all $\alpha\in\mathbb R$. Fix $R>0$, and set 
	$\mu_0=\frac12(\delta_{P_1}+\delta_{P_{-1}})$, $\mu_T=\delta_{Q_R}$.
	Suppose that $F(\cdot)\equiv\overline{B(0,1)}$, $L(\cdot)\equiv 1$, $\phi(\cdot)\equiv 1$, $G(x,\mu)=I_{\{\mu_T\}}(\mu)$ is the indicator function at $\mu_T$, $\psi(r)=|\log r|$.
	The problem \eqref{eq:prob} amounts to minimize
	\[\L_\psi(\eeta):=\int_{\Gamma_{I}}\int_I \left|\log\left(({\ev_t}_\sharp\eeta)(\{\gamma(t)\})\right)\right|\d t \d \eeta(\gamma)\]
	among all measures $\eeta\in\prob_2(\Gamma_I)$ supported on the set of $1$-Lipschitz continuous curves such that $e_0(\eeta)=\mu_0$ and $e_T(\eeta)=\mu_T$.

\begin{multicols}{2}
\noindent By symmetry, in order to optimize $\L_\psi$, the two groups of agents will try to meet as soon as possible at a point $Q_\alpha$ along the $x$-axis, $\alpha \in [0,R]$, and then proceed together to the destination. Given the time horizon $T$, such meeting point $Q_\alpha$ needs to verify $\sqrt{1+\alpha^2} + (R - \alpha) \leq T$. Hence, we seek for the minimum non-negative $\alpha$ satisfying $\sqrt{1+\alpha^2} + (R - \alpha) \leq T$.

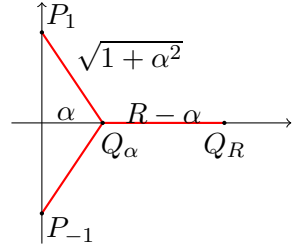
\begin{tikzpicture}[scale=0.8]
		
		\draw[->] (-0.5,0) -- (4.1,0);
		\draw[->] (0,-2) -- (0,2);
		
		\draw[red,thick] (0,1.5) -- (1,0);
		\draw[red,thick] (0,-1.5) -- (1,0);
		\draw[red,thick] (1,0) -- (3,0);
		
		\foreach \point in {(0,1.5),(0,-1.5),(1,0),(3,0)}{
			\fill[black] \point circle[radius=1pt];}
		
		\draw (-0.1,1.4) node[anchor=south west] {$P_1$};
		\draw (-0.1,-1.4) node[anchor=north west] {$P_{-1}$};
		
		\draw (3,0) node[anchor=north] {$Q_R$};
		\draw (0.8,0) node[anchor=north west] {$Q_\alpha$};
		
		\draw (0.4, 0.15) node {$\alpha$};
		\draw (2, 0.15) node {$R-\alpha$};
		
		\draw (0.4,0.7) node[anchor=south west] {$\sqrt{1+ \alpha^2}$};
	\end{tikzpicture}
	\end{multicols}

Respectively,
{\small
\begin{itemize}
		\item if $T<\sqrt{1+R^2}$, there are no solutions: we cannot find any $1$-Lipschitz continuous curve $\gamma_i$ connecting $P_i$ to $Q_R$ within time $T$, for $i = 1,2$,
		\item if $T \geq \sqrt{1+R^2}$, a minimizer is $\eeta=\frac{1}{2}(\delta_{\gamma_1}+\delta_{\gamma_2})$, with
        $\gamma_i(t)=P_i+t\dfrac{Q_{\bar\alpha}-P_i}{|Q_{\bar\alpha}-P_i|}$ for $t\in [0,\bar s]$, $\gamma_i(t)=Q_{\min \{\bar\alpha + t-\bar s, R\}}$ for $t\in[\bar s,T]$,
		where $\bar\alpha:=\max \left\{0, \frac{1-(T-R)^2}{2(T-R)}\right\}$ and $\bar{s}:= \sqrt{1 + \bar\alpha^2}$. In particular, for any $T \geq R+1$ we have $\bar\alpha = 0$ : the two populations move along the $y$-axis until they join at $(0,0)$, then they move together towards $Q_R$ along the $x$-axis.
	\end{itemize}
}
	
	Notice that, differently for the cases analyzed in the literature (cf. \cite{BCM}) where the multiplicity was invariant by time reparametrizations, here the position of the point $Q_\alpha$ where the trajectories first meet depends on the time horizon $T$.
\end{example}

\section{Existence of minimizers}\label{sec:existence-opt}
The following lemma is a crucial ingredient to prove the existence of minimizers in \eqref{eq:prob}. A similar result is Theorem 1.1 of \cite{FKZ}, but with a gap in the proof. 

\begin{lemma}\label{lemma:semicont}
	Let $\mathbb X$ be a separable metric space, $\{g_n\}_{n\in\mathbb N}$ be a sequence of Borel functions $g_n:\mathbb X\to [0,+\infty]$, and $\{\theta_n\}_{n\in\mathbb N}\subseteq \mathscr P(\mathbb X)$, $\theta\in\mathscr P(\mathbb X)$ 
	be such that $\theta_n\weakto\theta$ as $n\to +\infty$.
	Then,
\[\liminf_{n\to +\infty}\int_{\mathbb X} g_n(x)\d\theta_n(x)\ge \int_{\mathbb X} g(x)\d\theta(x),\]
where $g(x):=\displaystyle\min\left\{\liminf_{n\to +\infty}g_n(x),\liminf_{\substack{n\to +\infty\\ y\to x}}g_n(y)\right\}$.
\end{lemma}
\begin{proof}
	For any $m\in\mathbb N$, define 
\[u_m(x):=\sup\Big\{\inf\{g_n(y):\, y\in U,\,n\in\mathbb N,\,n\ge m\}:\,U\subseteq \mathbb X,\,U\textrm{ open},\,x\in U\Big\}.\]
	We prove that $u_m$ is l.s.c. for any $m\in\N$. For any $\varepsilon>0$ there exists an open neighbourhood $U_\varepsilon$ of $x$ such that, for all $z\in U_\varepsilon$,
	\begin{align*}
		u_m(x)&-\varepsilon\le\inf\{g_n(y):\,y\in U_\varepsilon,\,n\ge m\}\\
		\le&\sup\Big\{\inf\{g_n(y):\, y\in W,\,n\ge m\}:\,W\subseteq \mathbb{X},\,W\textrm{ open},\,z\in W\Big\}=u_m(z). 
	\end{align*}
	The ls.c. of $u_m(\cdot)$ follows taking the liminf for $z\to x$ and letting $\varepsilon\to 0^+$.
	\begin{equation*}
	\begin{split}
		&\sup_{m\in\mathbb N}u_m(x)=\\
		&=\sup\left\{\min\left\{\inf_{n\ge m} g_n(x),\inf_{n\ge m}\inf_{y\in U\setminus\{x\}} g_n(y)\right\}\right\}:\,\textrm{open }U\subseteq \mathbb X,\,x\in U,m\in\mathbb N\Big\}\\
		&=\sup_{m\in\mathbb N}\min\left\{\inf_{n\ge m} g_n(x),\sup_{\substack{U\textrm{ open}\\ x\in U}}\inf_{n\ge m}\inf_{y\in U\setminus\{x\}} g_n(y)\right\}\\
		&=\min\left\{\liminf_{n\to +\infty}g_n(x),\liminf_{\substack{n\to +\infty\\ y\to x}}g_n(y)\right\}=:g(x).
		\end{split}
	\end{equation*}
	Since $\{u_m\}_{m\in\mathbb N}$ is an increasing seq. of nonnegative Borel functions, with $\sup_m u_m=g$ and $u_m\le g_m$, by
	Monotone Convergence Thm and Lemma 5.17 in \cite{AGS}, we get
	\begin{align*}
		\int_{\mathbb X} g(x)\d\theta(x)=&\sup_{m}\int_{\mathbb X} u_m(x)\d\theta(x)\le \sup_m \liminf_{n\to +\infty}\int_{\mathbb X} u_m(x)\d\theta_n(x)\\
		\le&\liminf_{n\to +\infty}\int_{\mathbb X} u_n(x)\d\theta_n(x)\le\liminf_{n\to +\infty}\int_{\mathbb X} g_n(x)\d\theta_n(x).
	\end{align*}
\end{proof}

\begin{corollary}\label{cor:lscabs}
	Let $\mathbb X$ be a separable metric space, $\Lambda$ be a separable topological space, $h:\mathbb X\times\Lambda\to [0,+\infty[$ be a sequentially l.s.c. map.
	Suppose that $\{(\theta_n,\lambda_n)\}_{n\in\N}\subseteq \prob(\mathbb X)\times \Lambda$, $(\theta,\lambda)\in \prob(\mathbb X)\times\Lambda$, 
	satisfies $(\theta_n,\lambda_n)\to (\theta,\lambda)$ in $\prob(\mathbb X)\times\Lambda$ endowed with the product topology.
	Then
	\[\liminf_{n\to +\infty}\int_{\mathbb X}h(x,\lambda_n)\d\theta_n(x)\ge \int_{\mathbb X} h(x,\lambda)\d\theta(x).\]
\end{corollary}
\begin{proof}
	Defined, $g_n(x):=h(x,\lambda_n)$, $g_n(\cdot)$ is a Borel map and we have 
	\[
	\min\left\{\liminf_{n\to +\infty}g_n(x),\liminf_{\substack{n\to +\infty\\ y\to x}}g_n(y)\right\}\ge\min\left\{h(x,\lambda),\liminf_{\substack{n\to +\infty\\ y\to x}}h(y,\lambda_n)\right\} = h(x,\lambda).
	\]
	We then apply Lemma \ref{lemma:semicont}.
\end{proof}

\begin{proposition}\label{prop:lscinner}
	Let $\mathbb X, \mathbb Y, \mathbb W$ be separable metric spaces, $f:\mathbb X\times \prob(\mathbb W)\times [0,1]\to [0,+\infty]$ 
	be sequentially l.s.c. with $f(x,\theta,\cdot)$ non-increasing for every $(x,\theta)\in \mathbb X\times\prob(\mathbb W)$, $r\in C^0(\mathbb W;\mathbb Y)$, $C:\mathbb X\rightrightarrows \mathbb Y$ be an u.s.c. set valued map with closed non-empty images.
	Then the map $h_{r,C}:\mathbb X\times \prob(\mathbb W)\to [0,+\infty]$, defined by $h_{r,C}(x,\theta)=f(x,\theta,r_\sharp\theta(C(x)))$, is sequentially l.s.c.
\end{proposition}
\begin{proof}
	Fix $(x,\theta) \in \mathbb X\times \prob(\mathbb W)$. Given an open $V\supseteq C(x)$, by the u.s.c. of $C(\cdot)$ there exists an open $U_V\subseteq\mathbb X$, $x\in U_V$, such that $C(z)\subseteq V$ for all $z\in U_V$. In 
	particular, for any sequence $z_n\to x$ and $\theta_n \weakto \theta$, since $r_\sharp\theta_n\weakto r_\sharp\theta$, we can apply Portmanteau theorem (see e.g. Theorem 2.1 in \cite{BIL}) to get
	\[\limsup_{n\to +\infty}r_\sharp\theta_n(C(z_n))\le \limsup_{n\to +\infty}r_\sharp\theta_n(\overline V)\le r_\sharp\theta(\overline V).\] 
	By choosing a decreasing  sequence $\{V_j\}_{j\in\mathbb N}$ such that $V_j\supseteq C(x)$ for all  $j\in\mathbb N$ and $\displaystyle\bigcap_{j\in\mathbb N} \overline{V_j}=C(x)$, 
we obtain $\,\displaystyle\limsup_{n\to +\infty}r_\sharp\theta_n(C(z_n))\le r_\sharp\theta(C(x))$. 
	Hence,
	\begin{align*}
		&\liminf_{n \to +\infty} h_{r,C}(z_n,\theta_n)=\liminf_{n \to +\infty} f(z_n,\theta_n,r_\sharp\theta_n(C(z_n)))\\
		&\ge f\left(x,\theta,\limsup_{n \to +\infty}r_\sharp\theta_n(C(z_n))\right)
		\ge f(x,\theta,r_\sharp\theta(C(x)))=h_{r,C}(x,\theta).
	\end{align*}
\end{proof}

\begin{theorem}\label{thm:butlsc}
	Let $I=[a,b]$, $a\le b$, and $f:I\times\mathscr P(\Gamma_I)\times \R^d\times \R^d\times [0,1]\to [0,+\infty]$ be a Borel map satisfying: 
		$f(t,\cdot,\cdot,\cdot,\cdot)$ is sequentially l.s.c. for a.e. $t\in I$, 
		$f(t,\eeta,y,\cdot,s)$ is convex,
		$f(t,\eeta,y,z,\cdot)$ is non-increasing.
Fix $q\in [1,+\infty[$. Then,
\begin{itemize}
		\item[i)] the functional $K:\Gamma_I\times L^q(I;\R^d)\times \prob(\Gamma_I)\to [0,+\infty]$, with
		\[
		K(u,v,\eeta):=\int_I f(t,\eeta,u(t),v(t),{\ev_t}_\sharp\eeta(\{u(t)\}))\d t,
		\]
		is sequentially l.s.c. w.r.t. the strong topology on $\Gamma_I$, the weak topology on $L^q(I;\R^d)$ and the narrow topology on $\prob(\Gamma_I)$;
		\item[ii)] the map $\displaystyle\eeta\mapsto \int_{\Gamma_I}\int_I f(t,\eeta,\gamma(t),\dot\gamma(t),{\ev_t}_\sharp\eeta(\{\gamma(t)\}))\d t\d\eeta(\gamma)$
		is sequentially l.s.c. on the space $\prob(W^{1,q}(I; \R^d))$.
	\end{itemize}
	\end{theorem}
\begin{proof}
	Let $h:I\times \prob(\Gamma_I)\times \R^d\times \R^d\to [0,+\infty]$ be defined by $h(t,\eeta,y,z):=f(t,\eeta,y,z,e_{t\sharp}\eeta(\{y\}))$.
	Then, $h(\cdot)$ is a Borel map such that $h(t,\eeta,y,\cdot)$ is convex, and by Proposition \ref{prop:lscinner} with $\mathbb X=\R^d\times\R^d$, $\mathbb W=\Gamma_I$, $\mathbb Y=\R^d$, $C(x_1,x_2)=\{x_1\}$, $r=\ev_t$,
	we have that $h(t,\cdot,\cdot,\cdot)$ is sequentially l.s.c.
	By Theorem 2.3.1 and Remark 2.2.6 in \cite{But}, the functional $K$ is sequentially l.s.c. with respect to the desired topologies. Indeed, it is sufficient to consider sequences $\{\eeta_n\}_n\subset \prob(\Gamma_I)$ s.t. $\eeta_n\weakto\eeta$ and $\{u_n\}_n\subset\Gamma_I$ s.t. $u_n\to u$, and $\{v_n\}_n\subset L^q(I;\R^d)$ s.t. $v_n\weakto v$. Then apply the cited results in \cite{But} to the sequence $(\xi_n,v_n)\in L^p(I;\prob(\Gamma_I)\times\R^d)\times L^q(I;\R^d)$, where $\xi_n(t):=(\eeta_n,u_n(t))$.
	We now prove $ii)$.	We can write
	\[
	\int_{\Gamma_I}\int_I f(t,\eeta,\gamma(t),\dot\gamma(t),{\ev_t}_\sharp\eeta(\{\gamma(t)\}))\d t\d\eeta(\gamma)=\int_{\Gamma_I} K(\gamma,\dot\gamma,\eeta)\d\eeta(\gamma)
	\]
	and notice that for all $\gamma\in W^{1,q}(I; \R^d)$ and $\eeta\in\prob(W^{1,q}(I; \R^d))$ the map $(\gamma,\eeta)\mapsto K(\gamma,\dot\gamma,\eeta)$ is sequentially l.s.c. by $(i)$.
	We conclude by Corollary \ref{cor:lscabs}, with $\mathbb X=W^{1,q}(I; \R^d)$ and $\Lambda=L^q(I;X)\times\prob(W^{1,q}(I; \R^d))$ endowed with the product weak-narrow topology.
\end{proof}

\begin{proposition}\label{prop:lsc}
	Let $q\in[1,+\infty[$. The functional $\E_{\psi,\phi}:\prob(W^{1,q}(I; \R^d))\to [0,+\infty]$ of Definition \ref{def:cost} is sequentially l.s.c. with respect to narrow convergence.
\end{proposition}
\begin{proof}
	Define the function	$f(t,\eeta,x,z,s) := \psi(s) \cdot L(t,x,z,{\ev_t}_\sharp\eeta)$.
	Then, $f$ satisfies the assumptions of Theorem \ref{thm:butlsc} and point $(ii)$ of the same Theorem provides the l.s.c. of $\mathscr L_\psi(\cdot)$. An analogous argument proves the l.s.c. of $\mathscr G_\phi(\cdot)$. 
\end{proof}

As a consequence of previous results, we get the following.
\begin{theorem}[Existence of minimizers]
	Let $T>0$ and $\mu\in\prob_2(\R^d)$. Then, the functional $\E_{\psi,\phi}(\cdot)$ admits a minimizer in $\boldsymbol S_{[0,T]}(\mu)$.
\end{theorem}
\begin{proof}
The results follows since, by Proposition \ref{prop:lsc}, $\E_{\psi,\phi}(\cdot)$ is sequentially l.s.c. over $\prob(W^{1,q}([0,T]; \R^d))$, for any $q \in [1,+\infty[$ and  $\boldsymbol S_{[0,T]}(\mu)$ is a $W_2$-compact subset of $\prob_2(\Gamma_{[0,T]})$ as assumend in Definition \ref{def:dynamics}.
\end{proof}

\section{Dynamic Programming Principle}\label{sec:DPP}
We conclude proving a Dynamic Programming Principle for the \emph{value function} associated with the optimization problem \eqref{eq:prob} and defined by $V_T:[0,T]\times\prob_2(\R^d)$,
\[V_T(t,\mu):=\inf_{\eeta\in \boldsymbol S_{[t,T]}(\mu)}\E_{\psi,\phi}(\eeta).\]

\begin{definition}[Restriction and concatenation]\label{def:conccurv}
	Let $I,I_1,I_2,J\subseteq\mathbb R$ be compact intervals, $b\in I$, satisfying $I=I_1\cup I_2$, $J\subseteq I$, $I_1\cap I_2=\{b\}$.
\begin{enumerate}
		\item Given $\gamma\in\Gamma_I$, we define the \emph{restriction} $\gamma|_J:J\to X$ by $\gamma|_J(t)=\gamma(t)$ for all $t\in J$.
		\item Given $\gamma_i\in\Gamma_{I_i}$, $i=1,2$, with $\gamma_1(b)=\gamma_2(b)$, we define the \emph{concatenation} $\gamma_1\oplus\gamma_2\in\Gamma_I$ of $\gamma_1$ and $\gamma_2$ to be 
		the unique $\gamma\in \Gamma_I$ s.t. $\gamma|_{I_i}=\gamma_i$, $i=1,2$.
		\item Given  $\eeta\in\prob(\Gamma_{I})$, we define the \emph{restriction} $\boldsymbol\eta|_J$ of $\boldsymbol\eta$ to $J$ by setting, for all $\varphi:\Gamma_J\to\R$ bounded Borel function, 
		$\displaystyle\int_{\Gamma_J}\varphi\d\boldsymbol\eta|_J=\int_{\Gamma_I}\varphi(\gamma|_J)\d\boldsymbol\eta(\gamma).$
		\item Given $\eeta_i\in\prob(\Gamma_{I_i})$, $i=1,2$, such that $(\ev_b)_\sharp\eeta_1=(\ev_b)_\sharp\eeta_2=:\mu_b$, we define the \emph{concatenation} $\eeta_1\oplus\eeta_2\in\prob(\Gamma_{I})$
		by setting, for all $\varphi:\Gamma_I\to\R$ bounded Borel function,
		\[\int_{\Gamma_{I}}\varphi\d(\eeta_1\oplus\eeta_2):=\int_X\iint_{\Gamma_{I_1}\times\Gamma_{I_2}}\varphi(\gamma_1\oplus\gamma_2)\d\eeta_{1,x}(\gamma_1)\d\eeta_{2,x}(\gamma_2)\d\mu_b(x),\]
		where $\{\eeta_{i,x}\}_{x\in X}\subset\prob(\Gamma_{I_i})$, $i=1,2$ are the Borel family of measures obtained by disintegrating $\eeta_i$ with respect to the evaluation map $\ev_b$ 
		(cf. Theorem 5.3.1 in \cite{AGS}). Since for $\mu_b$-a.e. $x\in \R^d$ and $\boldsymbol\eta_{i,x}$-a.e. $\gamma_i\in \prob(\Gamma_{I_i})$, $i=1,2$, we have $\gamma_1(b)=\gamma_2(b)=x$,
		the integrand is well defined.
	\end{enumerate}
	\end{definition}

The time depended nature of the involved functionals well behaves with respect to restriction and concatenation, as the following proposition shows.
\begin{proposition}\label{prop:cost+}
	Let $I,I_1,I_2,J\subseteq\mathbb R$ be compact intervals, $b\in I$, satisfying $I=I_1\cup I_2$, $J\subseteq I$, $I_1\cap I_2=\{b\}$.
	Then the following hold:
{\small	\begin{enumerate}
		\item $\mul{t}{y}{\eeta}=\mul{t}{y}{\eeta|_J}$, for any $t\in J$, $y\in \R^d$;
		\item\label{item2} $\L_\psi(\eeta_1\oplus\eeta_2)=\L_\psi(\eeta_1)+\L_\psi(\eeta_2)$ for all $\eeta_i\in\prob(\Gamma_{I_i})$, $i=1,2$, with $\ev_{b\sharp}\eeta_1=\ev_{b\sharp}\eeta_2$;
		\item\label{item3} $\G_\phi(\eeta_1\oplus\eeta_2)=\G_\phi(\eeta_2)$ for all $\eeta_i\in\prob(\Gamma_{I_i})$, $i=1,2$, with $\ev_{b\sharp}\eeta_1=\ev_{b\sharp}\eeta_2$ and $\max\,I=\max\,I_2$;
		\item $\L_\psi(\eeta)\ge\L_\psi(\eeta|_J)$, $\G_\phi(\eeta)=\G_\phi(\eeta|_{I_2})$, for any $\eeta\in\prob(\Gamma_I)$ with $\max\,I=\max\,I_2$.
	\end{enumerate}
	}
\end{proposition}

\begin{theorem}[Dynamic Programming Principle]\label{thm:dpp}
	Let $0\le\tau\le s\le T$, $\mu\in\prob_2(\R^d)$. Then, the following holds
{\small	\begin{align}\label{eq:dpp}
		V_T(\tau,\mu)=\inf_{\eeta\in\boldsymbol S_{[\tau,T]}(\mu)}\ell_{\eeta}(s),&&\ell_{\eeta}(s):=\left\{\mathscr L_{\psi}(\boldsymbol\eta|_{[\tau,s]})+V_T(s,{\ev_s}_\sharp\eeta)\right\}.
	\end{align}}
	Moreover, for any $\eeta\in \boldsymbol S_{[\tau,T]}(\mu)$ the map $s\mapsto \ell_{\eeta}(s)$ is nondecreasing in $[\tau,T]$, and it is constant if and only if $V_T(\tau,\mu)=\E_{\psi,\phi}(\eeta)$.
\end{theorem}

\begin{proof}
	For any $\varepsilon>0$ there exists $\eeta^\eps\in \boldsymbol S_{[\tau,T]}(\mu)$ such that
{\small
	\begin{align*}
		V_T&(\tau,\mu)+\eps\ge\E_{\psi,\phi}(\eeta^\eps)
		=\mathscr{L}_\psi(\boldsymbol\eta^\varepsilon|_{[\tau,s]}\oplus\boldsymbol\eta^\varepsilon|_{[s,T]} )+\mathscr G_\phi(\boldsymbol\eta^\varepsilon|_{[s,T]})\\
		=&\mathscr{L}_\psi(\boldsymbol\eta^\varepsilon|_{[\tau,s]})+\mathscr{L}_\psi(\boldsymbol\eta^\varepsilon|_{[s,T]})+\mathscr G_\phi(\boldsymbol\eta^\varepsilon|_{[s,T]})
		=\mathscr{L}_\psi(\boldsymbol\eta^\varepsilon|_{[\tau,s]})+\E_{\psi,\phi}(\boldsymbol\eta^\varepsilon|_{[s,T]})\\
		\ge&\mathscr{L}_\psi(\boldsymbol\eta^\varepsilon|_{[\tau,s]})+V_T(s,\ev_{s\sharp}\boldsymbol\eta^\varepsilon|_{[s,T]})=\mathscr{L}_\psi(\boldsymbol\eta^\varepsilon|_{[\tau,s]})+V_T(s,\ev_{s\sharp}\boldsymbol\eta^\varepsilon)
		=\ell_{\boldsymbol\eta^\varepsilon}(s).
	\end{align*}
}	Thus, we have $\displaystyle V_T(\tau,\mu)\ge \inf_{\eeta\in\boldsymbol S_{[\tau,T]}(\mu)}\ell_{\eeta}(s)$.
	\par\medskip\par
	We prove the converse inequality. Let $\boldsymbol\eta\in \boldsymbol S_{[\tau,T]}(\mu)$, $\tau\le s\le T$, and set $\boldsymbol\eta_1=\boldsymbol\eta|_{[\tau,s]}$. For any $\varepsilon>0$,
	there is $\boldsymbol\eta_2^\varepsilon\in S_{[s,T]}(\ev_{s\sharp}\boldsymbol\eta)$ such that $V_T(s,\ev_{s\sharp}\boldsymbol\eta)+\varepsilon\ge \E_{\psi,\phi}(\eeta_2^\eps)$.
	Notice that $\ev_{s\sharp} \eeta_2^\eps=\ev_{s\sharp} \boldsymbol\eta$ and $\boldsymbol\eta_1\oplus\boldsymbol\eta_2^\varepsilon \in \boldsymbol S_{[\tau,T]}(\mu)$.
	Thus,
\small{	\begin{align*}
		V_T(\tau,\mu)\le&\mathscr L_\psi(\boldsymbol\eta_1\oplus\boldsymbol\eta_2^\varepsilon)+\mathscr G_\phi(\boldsymbol\eta_1\oplus\boldsymbol\eta_2^\varepsilon)=\mathscr L_\psi(\boldsymbol\eta_1)+\mathscr L_\psi(\boldsymbol\eta_2^\varepsilon)+\mathscr G_\phi(\boldsymbol\eta_2^\varepsilon)\\
		=&\mathscr L_\psi(\boldsymbol\eta_1)+\E_{\psi,\phi}(\eeta_2^\eps)
		\le\mathscr L_\psi(\boldsymbol\eta_1)+V_T(s,\ev_{s\sharp}\boldsymbol\eta)+\varepsilon
	\end{align*}}
	By letting $\varepsilon\to 0^+$ and recalling the definition of $\boldsymbol\eta_1$, we obtain $V_T(\tau,\mu)\le\mathscr L_\psi(\boldsymbol\eta|_{[\tau,s]})+V_T(s,\ev_{s\sharp}\boldsymbol\eta)$.
\end{proof}

\end{document}